\RequirePackage{fix-cm}
\documentclass[smallcondensed]{svjour3}       
\smartqed  
\usepackage{amsmath} 
\usepackage{verbatim}
\usepackage{amssymb, float,empheq}
\usepackage{color}
\usepackage{graphicx,epsfig}
\graphicspath{{figures/}}
\usepackage[applemac]{inputenc}
\usepackage{graphicx}
\def\beq{\begin{equation}}
\def\eeq{\end{equation}}
\def\baq{\begin{eqnarray}}
\def\eaq{\end{eqnarray}}
\def\baqn{\begin{eqnarray*}}
\def\eaqn{\end{eqnarray*}}

\newcommand{\R}{\mathbb{R}}

\newcommand{\ball}{\mathbb{B}}

\begin{document}

\title { Lur'e dynamical systems with state-dependent set-valued feedback 
}


\author{          Ba Khiet Le 
}

\institute{ Instituto de Ciencias de la Ingenier\'ia  \at
             Universidad de O'Higgins, Rancagua, Chile
\\
\email{lebakhiet@gmail.com} 
}
\date{Received: date / Accepted: date}

\maketitle

\begin{abstract}
Using a new implicit discretization scheme,  we study   in  this paper the existence and uniqueness of strong solutions  for a class of Lur'e  dynamical systems  where the set-valued feedback depends  on both time and  state. This work is  a generalization of \cite{abc} where the time-dependent set-valued feedback is considered to acquire only weak solutions. Obviously, strong solutions and implicit discretization scheme are nice properties, especially for numerical simulation. We also provide some conditions such that the solutions are exponentially attractive. The obtained results can be used to study the time-varying Lur'e systems with errors in data. Our result is new even the  set-valued feedback depends only on the time.
 
\end{abstract}

\section{Introduction}

\indent  It is known that  Lur'e  dynamical systems have been studied intensively recently with  many applications can be found in   control theory, engineering and  applied mathematics (see, e.g., \cite{l} for a survey).  In general,  the systems consist of  a smooth ordinary differential equation  $\dot{x}=g(x,\lambda)$ with output $y=h(x,\lambda)$ and a static  single-valued feedback $\lambda=F(y)$. In order to describe  discontinuous changes of velocity more effectively, Lur'e systems with static  set-valued feedback  was firstly considered  in \cite{bg} with a special case and then largely analyzed  in \cite{Acary,abc0,abc,ahl,ahl2,al,bg2,bg1,brogliato,cs}.
 Let us also mention that  set-valued  Lur'e  systems can be recast into  other non-smooth mathematical models  \cite{bdla,bg2,Cojocaru,Grabowski,Gwinner,Gwinner1} such as complementarity systems, evolution
variational inequalities, projected  systems,  relay systems \ldots

 In this paper,  we study  the well-posedness for a class of  Lur'e  dynamical systems where the set-valued feedback has the form of normal cone to a moving closed, convex set which depends not only on the time but also on  the state. For more details, let be given  a function $f: [0,+\infty) \times \R^n \to  \R^n $, some  matrices $ B:  \R^m\to  \R^n ,C: \R^n\to  \R^m$, $D : \R^m\to  \R^m$ and a set-valued mapping $K: [0,+\infty) \times \R^n \rightrightarrows \R^m$.   Then we want to find an absolutely continuous function $x(\cdot): [0,+\infty) \to \R^n$ with given initial point $x_0\in \R^n$ such that
\begin{subequations}
\label{eq:tot}
\begin{empheq}[left={({\mathcal S})}\empheqlbrace]{align}
  & \dot{x}(t) = f(t,x(t))+B\lambda(t)\; {\rm a.e.} \; t \in [0,+\infty); \label{1a}\\
  & y(t)=Cx(t)+D\lambda(t),\\
  & \lambda(t) \in -N_{K(t,x(t))}(y(t)), \;t\ge 0;\\
  & x(0) = x_0,
\end{empheq}
\end{subequations}
where  $\lambda, y : [0,+\infty)\to \R^m$ are two unknown connected  mappings.  

In \cite{abc}, the authors studied the time-dependent case $K(t,x)\equiv K(t)$ with motivation for considering the viability control problem and  the  output regulation problem, particularly  in power converters. To the best of our knowledge, it can be considered as the first work which considers  non-static  set-valued feedbacks with non-zero $D$ (see, e.g., \cite{ahl,ahl2,al,bg,bg2,bg1,brogliato,cs} for  static cases). The state-dependent case is left as an open problem. A kind of {weak} solution (\cite[Theorem 3.1]{abc}) is obtained under the following assumptions  :\\

$(A1)$ The matrix $D$ is positive semidefinite, and there exists a symmetric positive
definite matrix $P$ such that ${\rm ker}(D+D^T)\subset {\rm ker}(PB-C^T)$;\\

$(A2)$ There exists a nonnegative locally essentially bounded function \\$\rho: [0,+\infty) \to [0,+\infty)$ such that
$$
\Vert f(t,x)-f(t,y) \Vert \le \rho(t) \Vert x-y \Vert, \;x, y\in \R^n;
$$

$(A3)$ For each $t \geq  0$, ${\rm rge}(C) \cap  {\rm rint(rge}(N^{-1}_{K(t)} + D)) \not = \emptyset $;\\

$(A4)$ For every $t \geq  0$ and each $v \in  {\rm rge}(C) \cap  {\rm rge}(N^{-1}_{K(t)} + D)$, it holds that ${\rm rge}(D +
D^T ) \cap  (N^{-1}_{C(t)} + D)^{-1}(v) \not = \emptyset $;\\

$(A5)$ It holds that ${{\rm rge}(D) \subseteq  {\rm rge}(C)}$ and $K  : [0,\infty ) \rightrightarrows  \R^m$ has closed and convex values
for each $t \geq  0.$ Also, the mapping $K \cap {\rm rge}(C)$ varies in an absolutely continuous
manner with time; that is, there exists a locally absolutely continuous function
$\mu  : [0,\infty ) \rightarrow  \R^+$ such that
$$
d_H(K(t_1)\cap {\rm rge}(C), K(t_2)\cap {\rm rge}(C)) \le \vert \mu(t_1)-\mu(t_2) \vert, \;\;t_1, t_2 \ge 0,
$$
where $d_H$ denotes the Hausdorff distance. 
 \begin{figure}[h!]
\begin{center}
\includegraphics[scale=0.5]{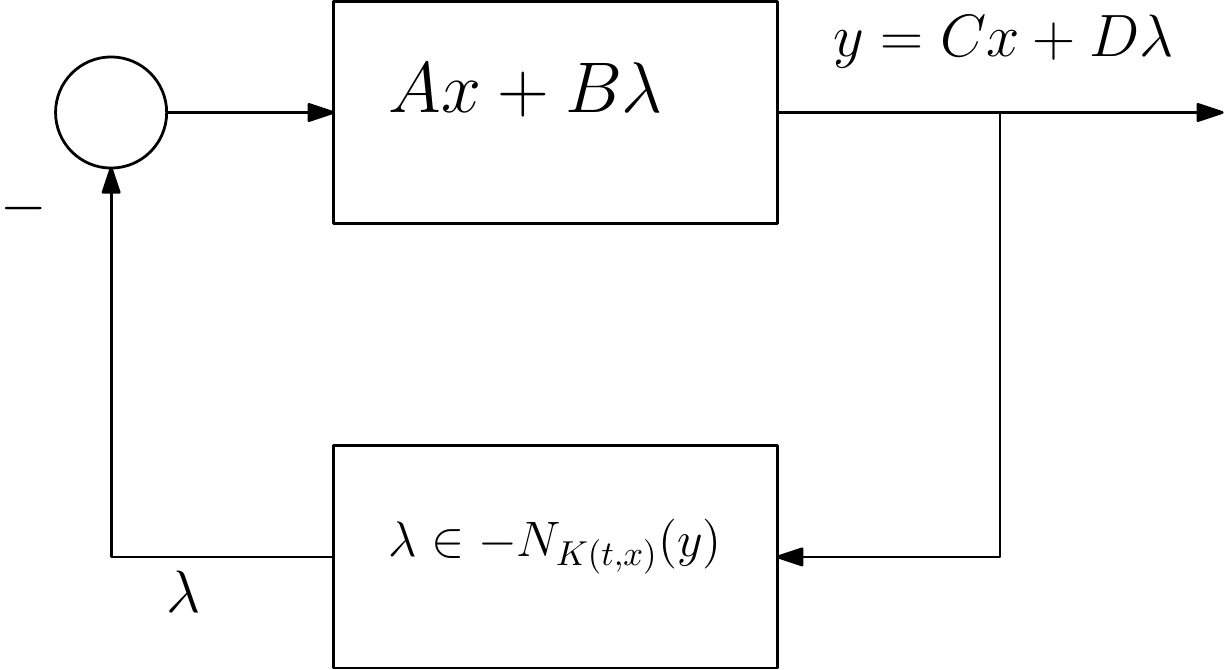}
\caption{ Lur'e  systems with state-dependent  feedback.}
\end{center}
\label{luref}
\end{figure}
The system $({\mathcal S})$ was rewritten into a time-varying first order differential inclusion where the right-hand side can be decomposed as a maximal monotone operator and a single-valued Lipschitz function to obtain weak solutions. 

The current paper generalizes   \cite{abc}  not only to the state-dependent moving set $K(t,x)$ but also to obtain strong solutions by using a new  implicit discretization scheme. Obviously,  strong solutions and the implicit discretization scheme  are desired properties which are advantages  for implementation in numerical simulations. In addition, we provide some conditions such that the solutions are exponentially attractive, i. e., the solutions converges to the origin with an exponential rate when the time is large. The obtained results can be used to study  time-varying Lur'e dynamical systems with errors in data.

The paper is organized as follows. In Section \ref{section2}, we recall   some notation and useful fundamental results. The well-posedness  and asymptotic behaviour of $(\mathcal{S})$  are  analyzed thoroughly in Section \ref{section3}. Application for the study of time-varying Lur'e dynamical systems with errors in data is presented in Section \ref{s4}. Some concluding remarks are given in Section \ref{sectionc}. 
\section{ Notation and Mathematical Backgrounds}
\label{section2}
Let us first introduce some notation that will be used in the sequel. Denote by $\langle\cdot,\cdot\rangle$ , $\|\cdot\|, \ball$ the  scalar product, the corresponding norm and the closed unit ball in Euclidean spaces. Let be given a closed, convex set $K\subset \R^n$. The distance and the projection from a point $s$ to $K$ are defined respectively by 
$${ d}(s,K):=\inf_{x\in K} \|s-x\|, \;\;{\rm proj}(s;K):=x \in K \;\;{\rm such \;that \;} { d}(s,K)= \|s-x\|.$$
The minimal norm element of $K$ is  defined by 
$$
K^0:= {\rm proj}(0; K).
$$
The Hausdorff distance between two closed, convex sets $K_1, K_2$ is given by 
$$d_H(K_1,K_2):=\max\{\sup_{x_1\in K_1} d(x_1,K_2), \sup_{x_2\in K_2} d(x_2,K_1)\}.$$
We define the $\it{indicator\; function}$ $i_K(\cdot)$   as follows
$$
i_K(x):=\left\{\begin{array}{lll}
0&{\rm if} &x\in K, \\ \\
+\infty&{\rm if} &x\notin K.
\end{array}\right.\;\;\;\;\;\;
$$
The $\it{normal\; cone}$ of a closed convex set $K$ is given by 
$$
N_K(x):=\partial i_K(x)=\{x^*\in H: \langle x^*, y-x \rangle \le 0,\;\forall y\in K\}.
$$

\begin{definition}
\noindent  A matrix $P\in \R^{n\times n}$  is called
 \begin{itemize}
\item  $positive$ $semidefinite$ if for all $x\in {\R^n},$ we have 
$$\langle Px,x \rangle \ge 0;$$
\item $positive$ $definite$ if there exists $\alpha>0$ such that for all $x\in {\R^n},$ we have 
$$
\langle Px,x \rangle \ge \alpha\|x\|^2;
$$
\item $symmetric$ if $P=P^T$, i.e., for all $x,y\in {\R^n},$ we have 
$$\langle Px,y\rangle=\langle   x,Py\rangle.$$
\end{itemize}
\end{definition}

We have the following fundamental lemma.
\begin{lemma}\label{estiD}
Let $D$ be a positive semidefinite  matrix. Then there exists some constant $c_1>0$ such that for all $x\in  {\rm rge}(D+D^T)$, we have:
\beq \label{constantd}
\langle Dx, x \rangle \ge c_1 \Vert x \Vert^2.
\eeq
\end{lemma}
\begin{remark}
Indeed, $c_1$ can be chosen as the small positive eigenvalue of $D+D^T$ if $D+D^T\neq 0$.
\end{remark}
Let be given some matrices $A\in \R^{n\times n}, B\in\R^{m\times n}, C\in \R^{m\times n} $ and $D\in \R^{m\times m}$.
\begin{definition} \label{passi} The system $(A,B,C,D)$ is $passive$ if there exists a symmetric positive definite matrix $P\in \R^{n\times n})$ such that  for all $x\in \R^n, y\in \R^m$, we have
\beq\label{passivebeq}
\langle PAx, x \rangle+ \langle (PB-C^T) y, x \rangle- \langle Dy,y \rangle \le 0,
\eeq
or equivalently, the matrix 
$$-\left( \begin{array}{cc}
 PA+A^TP& PB-C^T\\ \\
B^TP-C & -(D+D^T)
\end{array} \right)$$
is positive semidefinite. 
\end{definition}
We  provide a characterization for passive systems, see also \cite{cs} for another characterization.
\begin{lemma}\label{semi}
The matrix $D$ is positive semidefinite  and ${\rm ker}(D+D^T)\subset {\rm rge}(PB-C^T)$ for some symmetric positive definite matrix $P$ if and only if the system $(kI,B,C,D)$ is passive for some $k\in \R$.
\end{lemma}
\begin{proof}
$(\Leftarrow)$ See, e.g., \cite[Proposition 3]{cs} or  \cite[Lemma 1]{ahl}.\\
\noindent $(\Rightarrow)$  Since $D$ is positive semidefinite  there exists some $c_1>0$ such that for all $y\in \R^m$, we have  
$$
\langle Dy,y \rangle =  \langle Dy^{im},y^{im} \rangle\ge  c_1\|y^{im}\|^2,
$$
where $y^{im}$ is the orthogonal projection of $y$ onto ${\rm rge}(D+D^T)$. Similarly, there exists some $\alpha>0$ such that for all $x\in \R^n$
$$
\langle Px,x \rangle \ge \alpha \Vert x \Vert^2.
$$
\noindent We choose $k<0$ satisfying the inequality
\beq
2\sqrt{(-k)\alpha c_1}\ge \|PB-C^T\|.
\eeq
Then for all $x\in \R^n, y\in \R^m$, we obtain
\baqn
&&\langle (P(kI)x, x \rangle+ \langle (PB-C^T) y, x \rangle- \langle Dy,y \rangle\\
&=& k\langle Px, x \rangle+ \langle (PB-C^T) y^{im}, x \rangle- \langle D y^{im},y^{im} \rangle \\
&&\;\;(since\; {\rm ker}(D+D^T)\subset {\rm ker}(PB-C^T))\\
&\le& k\alpha\|x\|^2+\|PB-C^T\| \|x\| \|y_{im}\|-c_1\|y_{im}\|^2\le 0.
\eaqn
Thus $(kI,B,C,D)$ is passive.\\
  \qed
\end{proof}

\begin{definition}
A set-valued mapping $F: \R^n \rightrightarrows \R^n$ is called $\it{monotone}$ if for all $x,y\in \R^n,x^*\in F(x),y^*\in F(y)$, one has $\langle x^*-y^*,x-y\rangle \ge 0.$ In addition, it  is called $\it{maximal \;monotone}$ if there is no monotone  operator $G$ such that the graph of $F$ is contained strictly in  the graph of $G.$
\end{definition}
 \begin{proposition}{\rm (\cite{Aubin-Cellina,Brezis})}\label{Yosida}
Let $H$ be a Hilbert   space, $F:  {H} \rightrightarrows  {H}$ be a maximal monotone operator and let $\lambda>0$. Then\\
$1)$ the \emph{resolvent} of $F$  defined by $J_F^\lambda:=(I+\lambda F)^{-1} $ is a non-expansive and single-valued map from $ {H}$ to $ {H}$.\\
$2)$ the \emph{Yosida approximation} of $F$ defined by $F^\lambda:=\frac{1}{\lambda}(I-J_F^\lambda)=(\lambda I+F^{-1})^{-1}$ satisfies\\
\indent i) for all $x\in {H}$, $F^\lambda(x)\in F(J_F^\lambda x)$ ,\\
\indent ii) $F_\lambda$ is Lipschitz continuous with constant $\frac{1}{\lambda}$ and also maximal monotone.\\
\indent iii) If $x\in {\rm dom}(F)$, then $\|F^\lambda x\| \le \|F^0x\|$, where $F^0x$ is the element of $Fx$ of minimal norm.\\
\end{proposition}

\noindent  Let us recall Minty's Theorem in the setting of Hilbert  spaces  (see \cite{Aubin-Cellina,Brezis}).
\begin{proposition}\label{minty}
Let $H$ be a Hilbert   space.  Let $F: H\rightrightarrows H$ be a monotone operator. Then $F$ is maximal monotone if and only if ${\rm rge}(F+I)=H.$
\end{proposition}
Let be given two maximal monotone operators $F_1$ and
$F_2$, we recall  the definition of  pseudo-distance between  $F_1$ and $F_2$ introduced by Vladimirov \cite{Vladimirov} as follows
$$
{\rm dis}( F_1, F_2):=\sup\Big\{ \frac{\langle \eta_1 -\eta_2,z_2-z_1\rangle}{1+|\eta_1|+|\eta_2|}:  \eta_i\in F(z_i), z_i\in {\rm dom}\; (F_i), i=1,2\Big\}.
$$
\begin{lemma}\label{hauslem}
\cite{Vladimirov} If $ F_i=N_{A_i}$  where $A_i$ is a  closed convex set  $(i=1,2)$ then 
$$
{\rm dis}( F_1,  F_2)=d_H(A_1,A_2).
$$
\end{lemma}

\begin{lemma}\label{esti2mm} \cite{Kunze}
Let $F_1$, $F_2$ be two maximal monotone operators. For $\lambda>0, \delta >0$ and $x\in {\rm dom}(F_1)$, we have
\baqn
\Vert x - J^\lambda_{F_2}(x)\Vert &\le& \lambda \Vert F_1^0 x \Vert + {\rm dis}(F_1, F_2)+\sqrt{\lambda(1+\Vert F_1^0 x \Vert){\rm dis}(F_1, F_2)}\\
&\le& \lambda\Vert F_1^0 x \Vert + {\rm dis}(F_1, F_2)+ (\delta {\rm dis}(F_1, F_2)+\frac{\lambda(1+\Vert F_1^0 x \Vert)}{4\delta})\\
&\le& \frac{\lambda(1+(4\delta+1)\Vert F_1^0 x \Vert)}{4\delta})+(1+\delta){\rm dis}(F_1, F_2).
\eaqn
\end{lemma}
\begin{lemma}\label{closemm} \cite{Kunze}
Let $F_n$ be a sequence of maximal monotone operators in a Hilbert space $H$ such that ${\rm dis}(F_n,F)\to 0$ as $n\to+\infty$ for some maximal monotone operator $F$. Suppose that $x_n\in {\rm dom}(F_n)$ with $x_n\to x$ and that $y_n\in F_n(x_n)$ with $y_n\to y$ weakly for some $x, y\in H$. Then $x\in {\rm dom}(F)$ and $y\in F(x)$.
\end{lemma}
Let us end-up this section by recalling some versions of Gronwall's inequality.

\begin{lemma}\label{Gronwalldis}
Let $\alpha>0$ and $(u_n)$, $(\beta_n)$ be non-negative sequences satisfying 
\beq
u_n\le \alpha+\sum_{k=0}^{n-1}\beta_k u_k \;\;\forall n= 0,1,2,\ldots\;\; ({\rm with} \;\beta_{-1}:=0).
\eeq
Then, for all $n$, we have
$$u_n\le \alpha \;{\rm exp}\Big(\sum_{k=0}^{n-1}\beta_k\Big).$$
\end{lemma}

\begin{lemma}\label{gronwall}
Let $T>0$ be given and $a(\cdot),b(\cdot)\in L^1([0,T];\R)$ with $b(t)\ge 0$ for almost all $t\in [0,T].$ Let an absolutely continuous function $w: [0,T]\to \R_+$ satisfy
\beq
(1-\alpha)w'(t)\le a(t)w(t)+b(t)w^\alpha(t),\;\; a. e. \;t\in [0,T]
\eeq
where $0\le \alpha<1$. Then for all $t\in [0,T]$, we have 
\beq
w^{1-\alpha}(t)\le w^{1-\alpha}(0){\rm exp}\Big(\int_{0}^t a(\tau)d\tau\Big)+\int_{0}^t{\rm exp}\Big(\int_{s}^t a(\tau)d\tau\Big)b(s)ds.
\eeq
\end{lemma}
\section{Main results} \label{section3}
In this section, the well-posedness and asymptotic behaviour   of problem $(\mathcal{S})$ are studied. From $(1b)$ and $(1c)$ of $(\mathcal{S})$, it is easy to compute $\lambda(\cdot)$ in term of $x(\cdot)$:
$$\lambda(t)\in-({N}^{-1}_{K(t,x(t))}+D)^{-1}(Cx(t)),\;\;a.e. \;t\ge 0.$$
Therefore, we can rewrite the system $(\mathcal{S})$ in the form of first order differential inclusion as follows
\baq\nonumber
\dot{x}(t)&\in& f(t,x(t))-B(N^{-1}_{K(t,x(t))}+D)^{-1}(Cx(t))\\\label{main1}
&=& f(t,x(t))- B \Phi(t,x(t),x(t))\;\;a.e.\;\;t\ge 0,
\eaq
where 
\beq
\Phi(t,x,y):=(N^{-1}_{K(t,y)}+D)^{-1}Cx,\;\; t\ge 0 , x, y\in \R^n.
\eeq
\noindent    Suppose that the following assumptions hold.\\

\noindent$\mathbf{Assumption \;1}:$ The set-valued mapping 
$K: [0,+\infty) \times \R^n \rightrightarrows  \R^m$
 has non-empty, closed convex values 
such that  $K\cap {\rm rge}(C)$ has non-empty values and  there exist $L_{K1}\ge 0, 0 \le L_{K2}\le \frac{c_2}{\Vert C \Vert}$ such that  for all  $ t, s \ge 0$ and $x, y\in \R^n$ we have  \\
$$
{\rm d}_H(K(t,x) \cap {\rm rge}(C),  K(s,y) \cap {\rm rge}(C))\le   L_{K1} \vert t-s \vert +L_{K2} \Vert x-y \Vert,
$$
where $c_2>0$ is the smallest positive eigenvalue of $CC^T.$ \\
\noindent$\mathbf{Assumption \;2}:$ The matrix $D$ is positive semidefinite with  ${\rm rge}(D)\subset {\rm rge}(C)$ and 
 $${\rm ker}(D+D^T)\subset {\rm ker}(PB-C^T)$$
for some symmetric positive definite matrix $P$.\\

\noindent$\mathbf{Assumption \;3}:$ For all $t\ge 0$, if  $(N^{-1}_{K(t,y)}+D)^{-1}Cx \neq \emptyset$ for some $x, y\in \R^n$, it holds that $ {\rm rge}(D+D^T) \cap (N^{-1}_{K(t,y)}+D)^{-1}Cx\neq \emptyset$.\\

\noindent$\mathbf{Assumption \;4}:$  For all $t\in [0,T]$, $x\in \R^n:$ 
 $ {\rm rge}(C) \;\cap {\rm rint(rge}(N^{-1}_{K(t,x)}+D)) \neq \emptyset $.\\

\noindent$\mathbf{Assumption \;5}:$  The function $f: [0,+\infty) \times \R^n \to \R^n$ is continuous in the first variable and $L_f$-Lipschitz continuous in the second variable in the sense that there exist a continuous function $v_f: [0,+\infty) \to \R$ and $L_f>0$ such that for all  $s, t\ge 0$ and $x, y\in \R^n$ we have
$$
\Vert f(t,x)-  f(s,y)) \Vert \le  \vert v_f(t)-v_f(s) \vert +L_f \Vert x-y \Vert.
$$
\begin{remark}
 It is not difficult to relax that the moving set $K$ varies in an absolutely continuous way with respect to the time. For simplicity of calculation, we suppose that  $K$ moves  Lipschitz continuously  in both  time and  state. 
\end{remark}
Let be given some arbitrary real number $T>0$. The following lemmas are useful.
\begin{lemma}\label{cct}
There exists a constant $c_2>0$ such that for all $x\in {\rm rge}(CC^T)$, we have 
\beq
\langle CC^T x, x \rangle \ge c_2 \Vert x \Vert^2.
\eeq
\end{lemma}
\begin{proof}
It is easy to see that the matrix $CC^T$ is symmetric positive semidefinite and the conclusion follows thanks to Lemma \ref{estiD}. In addition, $c_2$ can be chosen as the smallest positive eigenvalue of $CC^T$ if $C\neq 0$.
\end{proof}
\begin{lemma}\label{chypo}
Let Assumption $1$ hold. Suppose that 
\beq
a_i\in {\rm N}_{K(t_i,x_i)}(b_i)\;\;\;{\rm for}\; \;a_i\in \R^m,  b_i \in {\rm rge}(C), x_i \in \R^n,\;   t_i \ge 0 \; (i=1,2).
\eeq
Then 
\beq 
\langle a_1-a_2, b_1-b_2 \rangle  \ge -  (\Vert a_1\Vert+ \Vert a_2 \Vert )({L}_{K1}\vert t_2-t_1 \vert+{L}_{K2}\Vert x_1-x_2 \Vert)
\eeq
where the constant numbers $L_{K1}$ and $L_{K2}>0$ are defined in Assumption $1$.
\end{lemma}
\begin{proof}
We have 
\beq
\langle a_1, z-b_1 \rangle \le 0, \;\;{\rm for\;all}\; z\in K(t_1,x_1).
\eeq
Note that  $b_2\in K(t_2,x_2)   \cap {\rm rge}(C)\subset K(t_1,x_1) \cap {\rm rge}(C)+({L}_{K1}\vert t_2-t_1 \vert+{L}_{K2}\Vert x_1-x_2 \Vert)\ball$. {Combining} with the last inequality, we obtain 
\beq \label{estim1}
\langle a_1, b_2-b_1\rangle \le ({L}_{K1}\vert t_2-t_1 \vert+{L}_{K2}\Vert x_1-x_2 \Vert) \Vert a_1 \Vert.
\eeq
Similarly, one has 
\beq \label{estim2}
\langle a_2, b_1-b_2 \rangle \le ({L}_{K1}\vert t_2-t_1 \vert+{L}_{K2}\Vert x_1-x_2 \Vert) \Vert a_2 \Vert.
\eeq
From (\ref{estim1}) and (\ref{estim2}), we deduce that 
\baqn
\langle a_1-a_2, b_1-b_2 \rangle \ge -(\Vert a_1\Vert+ \Vert a_2 \Vert )({L}_{K1}\vert t_2-t_1 \vert+{L}_{K2}\Vert x_1-x_2 \Vert),
\eaqn
and the conclusion follows. \qed
\end{proof}

\begin{lemma}\label{estif}
Let Assumption $5$ hold. Then there exists $\alpha_1>0$ such that 
\beq
\Vert f(t,x) \Vert \le \alpha_1(1+\Vert x \Vert ),\;\;{\rm for\;all}\;t\in [0,T], \;x\in \R^n.
\eeq
\end{lemma}
\begin{proof}
We have 
\beq
\Vert f(t,x) \Vert \le \Vert f(t,0) \Vert+L_f \Vert x  \Vert,
\eeq
and the conclusion follows with $\alpha_1:=\max \{\max_{t\in [0,T]} \Vert f(t,0) \Vert, L_f \}.$ \qed
\end{proof}

\begin{lemma}\label{Lipschitz}
Let Assumptions $1, 2, 3$ hold. Then there exist $\alpha_2, \alpha_3>0$ such that the single-valued minimal-norm  function 
$\Phi^0:  [0,T] \times \R^{2n}\to  {\rm rge}(D+D^T), (t,x,y)\mapsto \Phi^0(t,x,y)$ satisfies the following properties: \\

a) $\Vert \Phi^0(t,x,y) \Vert \le \alpha_2(1+\Vert x \Vert+\Vert y \Vert),\;\;\forall (t,x,y)\in  {\rm dom}(\Phi^0)$.\\

b) $\Vert \Phi^0(t_1,x_1,y_1) -\Phi^0(t_2,x_2,y_2)\Vert^2 \le \alpha_3 \Vert x_1-x_2 \Vert^2+ \alpha_3 (\Vert \Phi^0(t_1,x_1,y_1)\Vert+\Vert \Phi^0(t_2,x_2,y_2)\Vert)(\vert t_1-t_2 \vert+\Vert y_1-y_2  \Vert),\;\; \forall \;(t_i,x_i,y_i)\in  {\rm dom}(\Phi^0), i=1, 2$.
\end{lemma}
\begin{proof}
a) Given $(t,x,y)\in {\rm dom}(\Phi^0)$. Then $(N^{-1}_{K(t,y)}+D)^{-1}Cx \neq \emptyset$. Using Assumption $3$, we can find some $z_0\in {\rm rge}(D+D^T) \cap (N^{-1}_{K(t,y)}+D)^{-1}(Cx)={\rm rge}(D+D^T) \cap  \Phi(t,x,y)$. First, we prove that $\Phi^0(t,x,y)=z_0 \in {\rm rge}(D+D^T).$ Indeed, for each $z_1\in \Phi(t,x,y)$, it is sufficient to show that $\Vert z_1 \Vert \ge \Vert z_0\Vert$. We can write uniquely $z_1=z_1^{im}+z_1^{ker}$ where $z_1^{im}\in {\rm rge}(D+D^T), z_1^{ker}\in {\rm ker}(D+D^T)$ and $\langle z_1^{im}, z_1^{ker} \rangle=0$. One has
\beq\label{belongphi}
z_i\in (N^{-1}_{K(t,y)}+D)^{-1}(Cx) \Leftrightarrow z_i \in N_{K(t,y)}(Cx-Dz_i), \;i=0, 1.
\eeq
The monotonicity of $N_{K(t,y)}$ and $D$ allows us to deduce that $\langle D(z_0-z_1), z_0-z_1 \rangle =0$, or equivalently $z_1-z_0=z_1^{im}+z_1^{ker}-z_0 \in {\rm ker}(D+D^T)$. Therefore, $z_1^{im}-z_0 \in {\rm ker}(D+D^T) \cap {\rm rge}(D+D^T)=\{0\}$.
Consequently 
\beq
\Vert z_1 \Vert^2=  \Vert z_1^{im}\Vert^2+\Vert z_1^{ker} \Vert^2=\Vert z_0  \Vert^2+\Vert z_1^{ker} \Vert^2 \ge \Vert z_0  \Vert^2,
\eeq
and thus, we have $\Phi^0(t,x,y)=z_0 \in {\rm rge}(D+D^T)$. 

 Now, fix $(0,x_0,x_0)\in {\rm dom}(\Phi^0)$, where $x_0$ is an initial point of problem $(\mathcal{S})$.  Similarly as in (\ref{belongphi}) and  using Lemma \ref{chypo}, one obtains 
 \baq\nonumber
&& \langle C(x-x_0), \Phi^0(t,x,y)-\Phi^0(0,x_0,x_0) \rangle \\\nonumber
 &\ge&  \langle D(\Phi^0(t,x,y)-\Phi^0(0,x_0,x_0)), \Phi^0(t,x,y)-\Phi^0(0,x_0,x_0)\rangle\\\nonumber
& - &  (\Vert\Phi^0(t,x,y)\Vert+\Vert \Phi^0(0,x_0,x_0)\Vert) (tL_{K1}+   L_{K2}\Vert y-x_0 \Vert)\\\nonumber
 &\ge& c_1  \Vert\Phi^0(t,x,y)-\Phi^0(0,x_0,x_0)\Vert^2\\
 &-&(\Vert\Phi^0(t,x,y)\Vert+\Vert \Phi^0(0,x_0,x_0)\Vert) (TL_{K1}+L_{K2}\Vert y-x_0 \Vert),
 \label{est1t2}
 \eaq
 where $c_1>0$ is defined in Lemma \ref{estiD}.
  Thus we can find some $\beta>0$ such that 
\baqn
  \Vert\Phi^0(t,x,y)\Vert^2 &\le& \Vert\Phi^0(t,x,y)\Vert (\beta \Vert x \Vert+\beta \Vert y \Vert+ \beta) + \beta (\Vert x \Vert+\beta \Vert y \Vert+1)\\
  &\le& \beta(\Vert\Phi^0(t,x,y)\Vert+1)(\Vert x \Vert+ \Vert y \Vert+1)
\eaqn
 and the conclusion follows with $\alpha_2:=2\beta+1$. \\
 
 b) Similarly as in (\ref{est1t2}), for all $(t_i,x_i,y_i)\in  {\rm dom}(\Phi^0), i=1, 2$ we have
 \baq\nonumber
 && \langle C(x_1-x_2), \Phi^0(t_1,x_1,y_1)-\Phi^0(t_2,x_2,y_2) \rangle 
 \ge c_1  \Vert\Phi^0(t_1,x_1,y_1)-\Phi^0(t_2,x_2,y_2)\Vert^2\\
 &-&(\Vert\Phi^0(t_1,x_1,y_1)\Vert+\Vert \Phi^0(t_2,x_2,y_2)\Vert) (L_{K1}\vert t_1-t_2\vert+L_{K2}\Vert y_1-y_2 \Vert).
 \eaq
 Let us note that 
 \baqn
 && \langle C(x_1-x_2), \Phi^0(t_1,x_1,y_1)-\Phi^0(t_2,x_2,y_2) \rangle \\
  &\le& \frac{c_1}{2}\Vert\Phi^0(t_1,x_1,y_1)-\Phi^0(t_2,x_2,y_2)\Vert^2+\frac{\Vert C \Vert^2}{2c_1}\Vert x_1-x_2 \Vert^2,
 \eaqn
 and hence we obtain the conclusion. \qed
\end{proof}
\begin{lemma}\label{espera}
Suppose that $P\equiv I$, the identity matrix. Then for all $t, x, y \in {\rm dom}(\Phi)$, we have 
\baqn
B\Phi(t,x,y)&=&(B-C^T)\Phi(t,x,y)+C^T\Phi(t,x,y)\\
&=&(B-C^T)\Phi^0(t,x,y)+C^T\Phi(t,x,y).
\eaqn
\end{lemma}
\begin{proof}
It is sufficient to prove that $(B-C^T)\Phi$ is single-valued function and $(B-C^T)\Phi(t,x,y)=(B-C^T)\Phi^0(t,x,y)$. Let $z\in \Phi(t,x,y)$. Similarly as in the proof of Lemma \ref{Lipschitz}, we can write $z=\Phi^0(t,x,y)+z^{ker}$, where $z^{ker}$ is the projection of $z$ onto $\ker(D+D^T).$ Since $\ker(D+D^T)\subset (B-C^T)$, we have $(B-C^T)z=(B-C^T)(\Phi^0(t,x,y)+z^{ker})=(B-C^T)\Phi^0(t,x,y)$ and the proof is completed. 
\end{proof}

Let us recall the following result, which is firstly given in \cite{Le1} for $D=0$, see also \cite{abc}.
\begin{lemma}\label{haus}
Let be given two closed convex set $K_1, K_2$ such that $K_i \cap {\rm rge}(C) \neq \emptyset,\; {\rm rge}(D) \subset {\rm rge}(C)$ and let ${G}_i:=C^T(N_{K_i}^{-1}+D)^{-1}C, i= 1, 2$. Then 
\beq
{\rm dis}({G}_1, {G}_2) \le \frac{\Vert C \Vert}{c_2} d_H (K_1 \cap {\rm rge}(C), K_2 \cap {\rm rge}(C)),
\eeq
where $c_2>0$ is defined in Lemma \ref{cct}.
\end{lemma}
\begin{proof}  We have 
\baq \nonumber
&&{\rm dis}(G_1,G_2)\\\nonumber
&=&\sup\Big\{ \frac{\langle \eta_1 -\eta_2,z_2-z_1\rangle}{1+\|\eta_1\|+\|\eta_2\|}:  \eta_i\in C^T(N_{K_i}^{-1}+D)^{-1}Cz_i,  i=1,2\Big\}\\\nonumber
&=&\sup\Big\{ \frac{\langle C^T\mu_1 -C^T\mu_2,z_2-z_1\rangle}{1+\|C^T\mu_1\|+\|C^T\mu_2\|}:  \mu_i\in (N_{K_i}^{-1}+D)^{-1}Cz_i, i=1,2\Big\}\\\nonumber
&=&\sup\Big\{ \frac{\langle \mu_1 -\mu_2,Cz_2-Cz_1\rangle}{1+\|C^T\mu_1\|+\|C^T\mu_2\|}:  \mu_i\in N_{K_i}(Cz_i-D\mu_i),  z_i, i=1,2\Big\}\\ \nonumber
&\le&\sup\Big\{ \frac{\langle \mu_1 -\mu_2,(Cz_2-D\mu_2)-(Cz_1-D\mu_1)\rangle}{1+\|C^T\mu_1\|+\|C^T\mu_2\|}:  \mu_i\in N_{K_i}(Cz_i-D\mu_i),  z_i, i=1,2\Big\}, \label{est1}
\eaq
since $D$ is positive semidefinite. 

Let $w_1:={\rm proj}(Cz_2-D \mu_2,K_1\cap {\rm rge}(C))$ and $w_2:={\rm proj}(Cz_1-D\mu_1,K_2\cap {\rm rge}(C))$.  Then we have
\baq\nonumber
&&\langle \mu_1 ,(Cz_2-D\mu_2)-(Cz_1-D\mu_1)\rangle\\
&=&\langle \mu_1 ,Cz_2-D\mu_2-w_1\rangle+\langle \mu_1 ,w_1-(Cz_1-D\mu_1)\rangle\\\nonumber
&\le&\langle \mu_1 ,Cz_2-D\mu_2-w_1\rangle\;\;({\rm using \;the\;property\;of\;normal\;cone})\\\nonumber
&=&\langle \nu_1,Cz_2-D\mu_2-w_1\rangle\;\;\;{\rm where}\; \nu_1:={\rm proj}(\mu_1,  {\rm rge}(CC^T))\\\nonumber
&\le&\|\nu_1\| d_H(K_1\cap{\rm rge}(C)),K_2\cap{\rm rge}(C)))\\
&\le& \frac{\|C\|}{c_2}\|C^T \nu_1\| d_H(K_1\cap{\rm rge}(C)),K_2\cap{\rm rge}(C))) \label{est2x}
\eaq
where  the second equality holds since  $\mu_1-\nu_1\in {\rm ker}(CC^T)={\rm ker}(C^T)$, ${\rm rge(D)\subset rge(C)}$ and the third inequality is satisfied because 
$$
c_2\|\nu_1\|^2\le \langle CC^T \nu_1,\nu_1 \rangle\le \|C\| \|C^T \nu_1\| \|\nu_1\|.
$$
Similarly one has
\beq \label{est3x}
\langle \mu_2 ,(Cz_1-D\mu_1)-(Cz_2-D\mu_2)\rangle\le\frac{\|C\|}{c_2}\|C^T \nu_2 \| d_H(K_1\cap{\rm rge}(C)),K_2\cap{\rm rge}(C))) 
\eeq
  where $\nu_2:={\rm proj}(\mu_2,  {\rm rge}(CC^T)).$
From (\ref{est2x}) and  (\ref{est3x}), one has 
\baq \label{est4}\nonumber
&&\frac{\langle \mu_1 -\mu_2,(Cz_2-D\mu_2)-(Cz_1-D\mu_1)\rangle}{1+\|C^T\mu_1\|+\|C^T\mu_2\|}\\\nonumber
&\le&\frac{\|C\|}{c_2}\frac{\|C^T \nu_1\|+\|C^T \nu_2\|}{1+\|C^T \nu_1\|+\|C^T \nu_2\|} d_H(K_1\cap{\rm rge}(C)),K_2\cap{\rm rge}(C))\\
&\le& \frac{\|C\|}{c_2}  d_H(K_1\cap{\rm rge}(C)),K_2\cap{\rm rge}(C)),
\eaq
and the conclusion follows. \qed
\end{proof}
Now we are ready for the first  main result about the existence, uniqueness of strong solutions and the Lipschitz continuous dependence of solutions on the initial conditions. Let us define the admissible set 
\beq
\mathcal{A}:=\{x_0 \in \R^n:  (N^{-1}_{K(0,x_0)}+D)^{-1}Cx_0\neq \emptyset\}.
\eeq
\begin{theorem}\label{wpmain} (Existence)
Let Assumptions $1, 2, 3, 4, 5$ hold.
Then for each $x_0\in \mathcal{A}$, there exists a   solution $x(\cdot;x_0)$ defined on $[0,T]$ of   problem $(\mathcal{S})$ which is Lipschitz continuous. \end{theorem}
\begin{proof}
From Assumption 2, there exists $\kappa\in \R$ such that $(\kappa I, B, C, D)$ is passive by using Lemma \ref{semi}. 
By using change of variables, without loss of generality, we can suppose that $P\equiv I$, the identity matrix (see, e.g., \cite{ahl,cs}). 
 Let us use the following implicit scheme to approximate (\ref{main1}).\\
 
 Let be given some positive integer $n$. Let $h_n=T/n$ and $t^n_i=ih$ for $0\le i\le n.$ For $0\le i\le n-1$,  we can find the sequence $({x}^n_i)_{0\le i\le n}$ with $x^n_0=x_0$ as follows: 
 \begin{equation}\label{discrete}
\left\{\begin{array}{l}
y^n_i= {x}^n_{i}+h_nf(t^n_i,x^n_i)-h_n\kappa {x}^n_i \\ \\
{x}^n_{i+1}\in y_i^n -h_n F_{t^n_{i+1}, x^n_i}(x^n_{i+1}),
\end{array}\right.
\end{equation}
where $F_{t^n_{i+1},x^n_i}:=-\kappa I+B(N^{-1}_{K(t^n_{i+1},x^n_i)}+D)^{-1}C$ is a maximal monotone operator (see, e.g., \cite{cs,ahl,ahl2}). Then we can compute $x^n_{i+1}$ uniquely as follows
$$
x^n_{i+1}=  (I+h_n F_{t^n_{i+1}, x^n_i})^{-1}(y^n_i)= J^{h_n}_{F_{t^n_{i+1},x^n_i}}(y^n_i)
$$
where $J^\lambda_F$ denotes the {resolvent} of $F$ of index $\lambda$ which is non-expansive. 
Consequently, one can obtain the  algorithm to construct the sequences $(x_i^n)_{i=0}^n$  as follows.

$\mathbf{Algorithm}$

\noindent \texttt{Initialization.} Let $x^n_0:=x_0, y^n_0:=x^n_{0}+h_nf(t^n_0,x^n_0)-h_n\kappa x^n_0.$ \\

\noindent\texttt{Iteration.} For the  current points $x^n_i$ we can compute 
$$
y^n_i:=x^n_{i}+h_nf(t^n_i,x^n_i)-h_n\kappa x^n_i,
 $$
 and 
\beq\label{proje}
x^n_{i+1}:=J^{h_n}_{F_{t^n_{i+1},x^n_i}}(y^n_i).
\eeq
Clearly, the algorithm is well-defined and $x^n_{i+1}\in {\rm dom}(F_{t^n_{i+1},x^n_i})=  {\rm dom}(\Phi(t^n_{i+1},\cdot,x^n_i))$ for  $i=0,..,n-1$. 
On the other hand, using Lemma \ref{espera}, we can rewrite $(\ref{discrete})$ as follows
\baq\nonumber
x^n_{i+1} &\in& {x}^n_{i}+h_nf(t^n_i,x^n_i)+h_n\kappa (x^n_{i+1}-{x}^n_i)\\\nonumber
&&-h_n(B-C^T)\Phi^0(t^n_{i+1}, x^n_{i+1},x^n_i)-h_n G_{t^n_{i+1}, x^n_i}(x^n_{i+1})\\
&\in& z^n_i -h_n G_{t^n_{i+1}, x^n_i}(x^n_{i+1}),
\label{reduceG}
\eaq
where 
$$
z^n_i:= {x}^n_{i}+h_nf(t^n_i,x^n_i)+h_n\kappa (x^n_{i+1}-{x}^n_i)-h_n(B-C^T)\Phi^0(t^n_{i+1},  x^n_{i+1},x^n_i),
$$
and 
$G_{t^n_{i+1}, x^n_i}:=C^T\Phi(t^n_{i+1},  \cdot,x^n_i)C^T=(N^{-1}_{K(t^n_{i+1},x^n_i)}+D)^{-1}C$ is a maximal monotone operator with  ${\rm dom}(G_{t^n_{i+1},x^n_i})={\rm dom}(\Phi(t^n_{i+1},\cdot,x^n_i)$. Therefore, we can also compute the $x^n_{i+1}$ as follows 
\beq\label{computeG}
x^n_{i+1}=(I+h_n G_{t^n_{i+1}, x^n_i})^{-1}(z^n_i)= J^{h_n}_{G_{t^n_{i+1},x^n_i}}(z^n_i).
\eeq
Let us note that
\beq \label{estiG}
\Vert(G_{t, y})^0(x)\Vert \le \Vert C^T\Vert  \Vert\Phi^0(t,x,y)\Vert 
\le   \alpha_2\Vert C^T\Vert (1+ \Vert x\Vert+\Vert y\Vert),
\eeq
where $\alpha_2$ is defined in Lemma \ref{Lipschitz}. From (\ref{computeG}), we have
\baq\nonumber
\Vert x^n_{i+1}-x^n_{i}\Vert&=&\Vert J^{h_n}_{G_{t^n_{i+1},x^n_i}}(z^n_i)-x^n_{i}\Vert \\
&\le&\Vert J^{h_n}_{G_{t^n_{i+1},x^n_i}}(z^n_i)- J^{h_n}_{G_{t^n_{i+1},x^n_i}}(x^n_i) \Vert +\Vert J^{h_n}_{G_{t^n_{i+1},x^n_i}}(x^n_i)-x^n_i\Vert.
\label{est1}
\eaq
Since $J^{h_n}_{G_{t^n_{i+1},x^n_i}}$ is non-expansive, one has
\baq \nonumber
&&\Vert J^{h_n}_{G_{t^n_{i+1},x^n_i}}(z^n_i)- J^{h_n}_{G_{t^n_{i+1},x^n_i}}(x^n_i) \Vert\le \Vert  z^n_i- x^n_i\Vert \\\nonumber
&\le& h_n (\Vert f(t^n_i,x^n_i)\Vert+ \kappa \Vert x^n_{i+1}-x^n_i \Vert  + \Vert B-C^T \Vert  \Vert \Phi^0(t^n_{i+1}, x^n_i, x^n_{i+1})  \Vert  \\\nonumber
&\le& h_n ( \alpha_1(1+\Vert x^n_i \Vert)+\kappa (\Vert x^n_{i+1} \Vert+\Vert x^n_i \Vert)+  \alpha_2\Vert B-C^T \Vert   (1+\Vert x^n_i\Vert+ \Vert x^n_{i+1} \Vert)   \\\label{est2}
&\le &h_n (\alpha_1+ \alpha_2\Vert B-C^T \Vert +\kappa)(1+\Vert x^n_i \Vert+\Vert x^n_{i+1} \Vert).
\eaq
Let us chose some constant $\delta>0$ such that 
\beq\label{lktilde}
\tilde{L}_K:=\frac{(1+\delta)L_{K2} \Vert C \Vert}{c_2}<1.
\eeq
Note that $x^n_i\in {\rm dom}(G_{t^n_{i},x^n_{i-1}})$ for $i=0,..,n-1$ with $x^n_{-1}:=x^n_0$, by using Lemmas \ref{esti2mm}, \ref{haus}, Assumption 1 and (\ref{estiG}) we obtain 
\baq\nonumber
\Vert J^{h_n}_{G_{t^n_{i+1},x^n_i}}(x^n_i)-x^n_i\Vert&\le&  h_n\frac{1+(4\delta+1)\Vert G^0_{t^n_{i},x^n_{i-1}}(x^n_i)\Vert}{4\delta}+(1+\delta){\rm dis}(G_{t^n_{i+1},x^n_i}, G_{t^n_{i},x^n_{i-1}})\\\nonumber
&\le& h_n  \frac{1+(4\delta+1)\alpha_2\Vert C^T\Vert(1+\Vert x^n_i \Vert + \Vert  x^n_{i-1} \Vert)}{4\delta}\\
&+&\frac{(1+\delta) L_{K1}\Vert C \Vert}{c_2}h_n+\frac{(1+\delta)L_{K2} \Vert C \Vert}{c_2} \Vert x^n_i-x^n_{i-1}\Vert).
\label{est3}
\eaq
From (\ref{est1}), (\ref{est2}), (\ref{lktilde}) and (\ref{est3}), we can find some constant $\alpha_4>0$ such that
\baq\nonumber
\Vert x^n_{i+1}-x^n_{i}\Vert &\le& h_n\alpha_4(1+\Vert x^n_{i+1}\Vert +\Vert x^n_i \Vert +\Vert  x^n_{i-1} \Vert) )\\
&+&\tilde{L}_K \Vert x^n_i-x^n_{i-1}\Vert)
\eaq
where $\tilde{L}_K<1.$
Note that $x^n_{-1}:=x^n_0$, therefore we have 
\baq\label{estderi}
\Vert x^n_{i+1}-x^n_{i}\Vert &\le& h_n\alpha_4 \sum_{j=0}^i \tilde{L}^j_K(1+  \Vert x^n_{i-j+1} \Vert +  \Vert x^n_{i-j} \Vert +\Vert x^n_{i-j-1} \Vert )\\
&\le&  h_n\alpha_4(\frac{1}{1-\tilde{L}_K}+ \sum_{j=0}^i \tilde{L}^j_K( \Vert x^n_{i-j+1} \Vert +  \Vert x^n_{i-j} \Vert +\Vert x^n_{i-j-1} \Vert ).\nonumber
\eaq
Consequently 
\baqn
&&\Vert x^n_{i+1}-x^n_{0}\Vert \le  \sum_{j=0}^i  \Vert x^n_{j+1}-x^n_{j}\Vert \\
&\le& {h_n\alpha_4}(\frac{i+1}{ {1-\tilde{L}_K}}+ \Vert x^n_{i+1}\Vert+ 3\sum_{j=0}^i \tilde{L}^j_K\sum_{j=0}^i \Vert x^n_j \Vert)\\
&\le&  \frac{\alpha_4T}{ {1-\tilde{L}_K}}+{h_n\alpha_4}\Vert x^n_{i+1}\Vert +\frac{3h_n\alpha_4}{ {1-\tilde{L}_K}}\sum_{j=0}^i \Vert x^n_j \Vert.
\eaqn
We can choose $n$ large enough such that $h_n\alpha_4<1/2$. Then we  have 
$$
\Vert x^n_{i+1}\Vert\le \beta+\alpha_5h_n\sum_{j=0}^i \Vert x^n_j \Vert,
$$
where
$$
\beta:=2\Vert x_0\Vert+\frac{2\alpha_4T}{ {1-\tilde{L}_K}}, \;\;\alpha_5:=\frac{6\alpha_4}{ {1-\tilde{L}_K}}.
$$
 Using the discrete Gronwall's inequality, one has 
\beq
\Vert x^n_{i+1} \Vert \le M_1:=\beta e^{\alpha_5 T},\;i=0, 1, \ldots,n-1.
\eeq
Combining with (\ref{estderi}), we have  
\beq
\Vert\frac{x^n_{i+1}-x^n_{i}}{h_n}\Vert \le \frac{\alpha_4(1+3M_1)}{1-\tilde{L}_K}:=M_2 .
\eeq
We construct the sequences of functions $(x_n(\cdot))_n,$  $ (\theta_n(\cdot))_n,$ $(\eta_n(\cdot))_n$ on $[0,T]$   as follows:  on $[t^n_i,t^n_{i+1})$ for $0\le i \le n-1$, we set
\beq\label{funx}
x_n(t):=x^n_i+\frac{x^n_{i+1}-x^n_i}{h_n}(t-t^n_i), 
\eeq
and 
\beq\label{funtheta}
\theta_n(t):=t^n_i,\;\;\eta_n(t):=t^n_{i+1}.
 \eeq
Then, for all $t\in (t^n_{i},t^n_{i+1})$
$$\|\dot{x}_n(t)\|=\|\frac{x^n_{i+1}-x^n_i}{h_n}\|\le  M_2,$$
and 
 \beq\label{idapro}
 \sup_{t\in  [0,T]}\{|\theta_n(t)-t|,|\eta_n(t)-t|\}  \le h_n\to 0\; {\rm as}\;\;n\to +\infty.
 \eeq
  Consequently the sequence of functions $\big(x_n(\cdot)\big)_n$ is uniformly bounded and equi-Lipschitz.  Using  Arzel\`a--Ascoli theorem, there exist a Lipschitz function $x(\cdot): [0,T]\to \R^n$  and a subsequence, still denoted by  $\big(x_n(\cdot)\big)_n$, such that  
\begin{itemize}
\item $x_n(\cdot)$ converges strongly to $x(\cdot)$ in $\mathcal{C}([0,T];\R^n)$;
\item $\dot{x}_n(\cdot)$ converges weakly to $\dot{x}(\cdot)$ in $L^{2}([0,T];\R^n)$.
\end{itemize}
 In particular, $x(0)=x_0.$ In addition, from (\ref{reduceG}), (\ref{funx}) and (\ref{funtheta}) we obtain
\baq\nonumber
\dot{x}_n(t)&\in&f_n(t)+\kappa(x_n(\eta_n(t))-x_n(\theta_n(t)))\\\nonumber
&-&(B-C^T)\Phi^0(\eta_n(t),  x_n(\eta_n(t)),x_n(\theta_n(t)))\\
&-& G_{\eta_n(t), x_n(\theta_n(t)}(x_n(\eta_n(t))),
\label{appro}
\eaq
where $f_n(t):=f(\theta_n(t), x_n(\theta_n(t)))$ . We define the operators $ \mathcal{G}, \mathcal{G}_n: L^{2}([0,T];\R^n) \to L^{2}([0,T];\R^n)$ for each positive integer $n$ as follows
$$
w^*\in \mathcal{G}(w) \Leftrightarrow w^*(t)\in G_{(t, x(t))}(w(t)) \;a.e. \; t\in [0,T]
$$
and 
$$
w^*\in \mathcal{G}_n(w) \Leftrightarrow w^*(t)\in G_{\eta_n(t), x_n(\theta_n(t))}(w(t)) \;a.e. \; t\in [0,T].
$$
Using Minty's theorem, we can conclude that  $\mathcal{G}_n, \mathcal{G}$ are maximal monotone operators since for each $t\in [0,T]$, the operators $G_{t, x(t)}$ and $G_{\eta_n(t), x_n(\theta_n(t))}$ are maximal monotone. In addition, one has 
\baqn
&&{\rm dis}( \mathcal{G}_n,  \mathcal{G})\\
&=&\sup\Big\{\frac{  \int_0^T\langle z^*_n(t) -z^*(t),z_n(t)-z(t)\rangle dt}{1+\|z^*_n\|_{L^2}+\|z^*\|_{L^2}}:  z^*_n\in   \mathcal{G}_n(z_n), z^*\in     \mathcal{G}(z)\Big\}\\\nonumber\\
&\le &\sup\Big\{\frac{  \int_0^T{\rm dis}(G_{\eta_n(t), x_n(\theta_n(t))},G_{t,x(t)})(1+\|z^*_n(t) \|+\|z^*(t) \| ) dt}{1+\|z^*_n\|_{L^2}+\|z^*\|_{L^2}}:  z^*_n\in   \mathcal{G}_n(z_n), z^*\in     \mathcal{G}(z)\Big\}\\\nonumber
&& ({\rm \;using \;the\;definition \; of} \;{\rm dis}(G_{\eta_n(t), x_n(\theta_n(t))},G_{t,x(t)}))\\
&\le &\frac{\Vert C\Vert}{c_2}\sup\Big\{\frac{  \int_0^T(L_{K1}\vert\eta_n(t)-t \vert +L_{K_2}\Vert x_n(\theta_n(t)-x(t) \Vert)(1+\|z^*_n(t) \|+\|z^*(t) \|) dt}{1+\|z^*_n\|_{L^2}+\|z^*\|_{L^2}}: \\\nonumber
&& z^*_n\in   \mathcal{G}_n(z_n), z^*\in     \mathcal{G}(z)\Big\}\\\nonumber
&& {\rm (using\;Lemma \;\ref{haus}\; and \;Assumption\; 1)} \\
&\le &\frac{\Vert C\Vert}{c_2}(L_{K1} \Vert\eta_n-I \Vert_{L^2} +L_{K_2}\Vert x_n\circ \theta_n-x \Vert_{L^2})\sup\Big\{ \frac{1+\|z^*_n \|_{L^2}+\|z^* \|_{L^2}}{1+\|z^*_n \|_{L^2}+\|z^* \|_{L^2}}:\\\nonumber
&&z^*_n\in   \mathcal{G}_n(z_n), z^*\in     \mathcal{G}(z)\Big\}\\\nonumber
&= &\frac{\Vert C\Vert}{c_2}(L_{K1} \Vert\eta_n-I \Vert_{L^2} +L_{K_2}\Vert x_n\circ \theta_n-x \Vert_{L^2}) \to 0,
\eaqn
as $n\to +\infty.$

Using Assumption 5, Lemma \ref{Lipschitz} and the fact that $\dot{x}_n$ converges weakly to $\dot{x}$ in $L^{2}([0,T];\R^n)$, we have
$$
\dot{x}_n - f_n-\kappa (x_n \circ \eta_n -x_n \circ \theta_n)+(B-C^T)\Phi^0(\eta_n,  x_n \circ\eta_n, x_n \circ \theta_n)
$$
converges weakly in $L^{2}([0,T];\R^n)$ to 
$$
\dot{x} - f(\cdot,x)+(B-C^T)\Phi^0(\cdot, x, x).
$$
 On the other hand,  $x_n\circ \eta_n$ converges strongly $x$ in $L^{2}([0,T];\R^n)$. Combing with (\ref{appro}) and using Lemma \ref{closemm}, we deduce that 
 \beq
 \dot{x}- f(\cdot,x)+(B-C^T)\Phi^0(\cdot, x, x)\in -\mathcal{G}(x),
 \eeq
 or equivalently
 \beq
 \dot{x}(t)- f(t,x(t))+(B-C^T)\Phi^0(t, x(t), x(t))\in -{G}_{t,x(t)}(x(t)),\;\;a.e.\; t\in [0,T].
 \eeq
 Consequently, one has
 \baq\nonumber
 \dot{x}(t)&\in& f(t,x(t))-(B-C^T)\Phi^0(t, x(t), x(t)) -{G}_{t,x(t)}(x(t))\\\nonumber
 &=&f(t,x(t)) -(B-C^T)\Phi^0(t, x(t), x(t))-C^T(N^{-1}_{K(t,x)}+D)^{-1}Cx(t)\\
 &=&f(t,x(t)) - B(N^{-1}_{K(t,x)}+D)Cx(t),\;\;a.e.\; t\in [0,T],
 \eaq
 and the conclusion follows.
  \qed
\end{proof}
\begin{remark}
(i) Our  discretization method provides a feasible way to study the state-dependent Lur'e dynamical systems for the first time. In addition, it is remarkable that  the obtained solutions are strong.\\
(ii) If $B=C=I, D=0$ then problem $(\mathcal{S})$ becomes the well-known state-dependent sweeping process. Then $c_2=1$ and $L_{K_2}<\frac{c_2}{\Vert C\Vert }=1$, which is accordant with the  result developed in \cite{Kunze1}.   In addition, the authors in \cite{Kunze1} provided some examples to show that the existence of solutions may lack if $L_{K2}\ge 1$ and mentioned that we may not have the uniqueness of solutions even for  $L_{K2}< 1$. So the upper bound of $L_{K2}$ in Assumption 1 is optimal for our existence result.  \\
(iii) However for the uniqueness and Lipschitz dependence of solutions on initial conditions, we can obtain the positive answer thanks to the positive semidefiniteness of $D$, if   the moving set $K$ has a special form, namely it can be decomposed as a sum of a time-dependent moving set and a single-valued Lipschitz function. 
\end{remark}
\noindent$\mathbf{Assumption\; 1'}:$ Suppose that
$$
K(t,x)=K_1(t)+h(t,x),\;\;t\ge 0, x\in \R^n,
$$
where $K_1: [0,+\infty) \rightrightarrows  \R^m$ has non-empty, closed  convex values and $h:  [0,+\infty) \times \R^n  \to {\rm rge}(D+D^T)$ is a single-valued mapping. In addition,  there exist   $L_h, L_{h1}, L_{h2} \ge 0$ such that  for all  $s, t\ge 0$ and $x, y\in \R^n$, we have  \\
$$
{\rm dis}_H(K_1(t),  K_1(s))\le  L_{h1}\vert t-s \vert,
$$
$$
\Vert h(t,x)-h(s,y)\Vert\le  L_{h2}\vert t-s \vert+   L_h \Vert x-y \Vert.
$$
\begin{lemma}\label{norma}
Let Assumption 1' hold. Suppose that 
\beq\nonumber
a_i\in {\rm N}_{K(t,x_i)}(b_i)\;\;\;{\rm for}\; \;(a_i, b_i)\in \R^{2m}, x_i \in \R^n,\;   t_i \ge 0 \; (i=1,2).
\eeq
Then 
\beq 
\langle a_1-a_2, b_1-b_2 \rangle  \ge  \langle a_1-  a_2 ,  h(t,x_1)-h(t,x_2) \rangle.
\eeq
\end{lemma}
\begin{proof}
We have
$$
a_i\in {\rm N}_{K(t,x_i)}(b_i)= {\rm N}_{K_1(t)+f(t,x_i)}(b_i)={\rm N}_{K_1(t)}(b_i-f(t,x_i)).
$$
Since the normal cone of a convex set is monotone, we deduce that
$$
\langle a_1-a_2, b_1-h(t,x_1)-b_2+h(t,x_2) \rangle\ge 0,
$$
and the conclusion follows.
 \qed

\end{proof}

\begin{theorem} (Uniqueness and Lipschitz dependence on initial condition)\label{uniqueness}
Let  Assumptions $1', 2, 3, 4, 5$  hold.
Then for each $x_0\in\R^n$ such that $x_0\in \mathcal{A},$ problem $(\mathcal{S})$ has  a unique  solution $x(\cdot;x_0)$ on $[0,T]$. In addition, the mapping $x_0\mapsto x(\cdot;x_0)$ is Lipschitz continuous.
\end{theorem}
\begin{proof}
\noindent It is easy to see that Assumption 1' implies Assumption 1, so the existence of solutions is obtained.  Now, let $x_i$ be a solution of $(\mathcal{S})$ with the  initial condition $x_i(0)=x_{i0}, \;i= 1, 2$. We have 

\begin{equation}\label{uni}
\left\{\begin{array}{l}
\dot{x}_i(t)= f(t,x_i(t))-By_i(t),\\ \\
y_i(t) \in (N^{-1}_{K(t,x_i(t))}+D)^{-1}(Cx_i(t)),\;\;a.e.\;t\in [0,T].
\end{array}\right.
\end{equation}
The inclusion in (\ref{uni}) is  equivalent to 
\baqn
y_i(t) \in N_{K(t,x_i(t))}(Cx_i(t)-Dy_i(t)).
\eaqn
 Using  Lemma \ref{norma}, we obtain that 
\baqn
&&\langle y_1(t)-y_2(t), (Cx_1(t)-Dy_1(t)) -(Cx_2(t)-Dy_2(t)) \rangle\\
&\ge & \langle y_1(t)-y_2(t), h(t,x_1(t))-h(t,x_2(t)) \rangle \\
&=& \langle y^{im}_1(t)-y^{im}_2(t), h(t,x_1(t))-h(t,x_2(t)) \rangle \;({\rm since}\;{\rm rge}(h)\subset {\rm rge}(D+D^T))\\
&\ge & -L_h \Vert  y^{im}_1(t)-y^{im}_2(t) \Vert \Vert x_1(t)-x_2(t) \Vert,
\eaqn 
where $y^{im}$ denotes the projection of $y$ onto ${\rm rge}(D+D^T)$. 
Therefore
\baq\nonumber 
&&\langle y_1(t)-y_2(t), Cx_1(t)-Cx_2(t)) \rangle \\\nonumber
&\ge&  \langle y_1(t)-y_2(t), Dy_1(t)-Dy_2(t) \rangle -L_h \Vert  y^{im}_1(t)-y^{im}_2(t) \Vert \Vert x_1(t)-x_2(t) \Vert\\
&\ge& c_1  \Vert y^{im}_1(t)-y^{im}_2(t) \Vert^2 -L_h \Vert  y^{im}_1(t)-y^{im}_2(t) \Vert \Vert x_1(t)-x_2(t) \Vert,
\label{estirela}
\eaq
 where $c_1>0$ is defined in Lemma \ref{estiD}.  Hence
\baqn
&&\langle By_1(t)-By_2(t), x_1(t)-x_2(t) \rangle\\
&=&\langle y_1(t)-y_2(t), (Cx_1(t)-Cx_2(t)) \rangle + \langle (B-C^T)(y_1(t)-y_2(t),  x_1(t)-x_2(t) \rangle\\
&\ge &  c_1  \Vert y^{im}_1(t)-y^{im}_2(t) \Vert^2 -L \Vert  y^{im}_1(t)-y^{im}_2(t) \Vert \Vert x_1(t)-x_2(t) \Vert\\
&&({\rm where}\; L:=L_h+ \Vert B-C^T \Vert)\\
&\ge & \frac{-L^2}{4c_1} \Vert x_1(t)-x_2(t) \Vert ^2\;({\rm use\;the\;inequality\;} a^2+b^2\ge 2ab,\;\forall \;a, b \in \R).
\eaqn 
Consequently, we have
\baqn
\frac{d}{dt}  \frac{1}{2}\Vert x_1(t)-x_2(t) \Vert ^2&=&\langle  \dot{x}_1(t)-\dot{x}_2(t),  x_1(t)-x_2(t) \rangle\\
&=&   \langle f(t,x_1(t))-f(t,x_2(t)) - (By_1(t)-By_2(t)),  x_1(t)-x_2(t) \rangle\\
&\le & (L_f+ \frac{L^2}{4c_1}) \Vert  x_1(t)-x_2(t) \Vert ^2=\gamma \Vert  x_1(t)-x_2(t) \Vert ^2,
\eaqn
where $\gamma:=L_f+ \frac{L^2}{4c_1}$.
Using Gronwall's inequality, we obtain that 
$$
\Vert x_1(t)-x_2(t) \Vert \le \Vert x_1(0)-x_2(0) \Vert e^{\gamma t}\le  \Vert x_{10}-x_{20} \Vert e^{\gamma T}, \forall \; t\in [0,T],
$$
 and the conclusion follows.\qed
\end{proof}
\begin{remark}
Since $T>0$ is arbitrary, one can define the unique solution $x(\cdot;x_0)$ of problem $(\mathcal{S})$ on $[0,+\infty)$. Now we are interested in the asymptotic behaviour of the problem $(\mathcal{S})$, i.e., the behaviour of solutions when the time is large. 
\end{remark}
\begin{theorem} (Globally exponential attractivity)
Let all the assumptions of Theorem \ref{uniqueness} hold. In addition, suppose that 
\beq
\langle f(t,x), x \rangle \le - \sigma \Vert x \Vert^2, \;\;0\in  h(t,0)+  K_1(t)=K(t,0),\;\;\forall t\ge 0, x\in \R^n,
\eeq
for some $\sigma > \frac{(L_h+ \Vert B-C^T \Vert)^2}{4c_1}$. 
Then the unique solution $x(\cdot)$ of $(\mathcal{S})$ starting at a given point $x_0$ exponentially converges to the origin when the time is large, i.e., 
$$
\Vert x(t) \Vert \le e^{-\delta t} \Vert x_0 \Vert \to 0\;\;{\rm as}\;\; t\to +\infty,
$$
where $\delta:=\sigma - \frac{(L_h+ \Vert B-C^T \Vert)^2}{4c_1}>0$.
\end{theorem}
\begin{proof}
The unique solution $x(\cdot)$ satisfies 
$$
\dot{x}(t)= f(t,x(t))-By(t), \; y(t)\in (N^{-1}_{K(t,x(t))}+D)^{-1}(Cx(t)),\;\;a.e.\;t\ge 0.
$$
Then 
\baqn
y(t) &\in& N_{K(t,x(t))}(Cx(t)-Dy(t))=N_{K_1(t)+h(t,x(t))}(Cx(t)-Dy(t))\\
&=&N_{K_1(t)}(Cx(t)-Dy(t)-h(t,x(t))).
\eaqn
Since $ -h(t,0)\in  K_1(t)$, we have 
$$
\langle y(t), Cx(t)-Dy(t)-h(t,x(t))+h(t,0) \rangle\ge 0.
$$
Thus 
\baqn
\langle y(t), Cx(t) \rangle&\ge& \langle y(t),Dy(t)\rangle+\langle y(t),h(t,x(t))-h(t,0)\rangle\\
&\ge& c_1 \Vert y^{im}(t)\Vert^2+ \langle y^{im}(t),h(t,x(t))-h(t,0)\rangle\;\;({\rm since}\;{\rm rge}(h)\subset {\rm rge}(D+D^T))\\
&\ge& c_1 \Vert y^{im}(t)\Vert^2-L_h\Vert y^{im}(t)\Vert \Vert x(t)\Vert.
\eaqn
Note that
\baqn
\frac{d}{dt}(\frac{1}{2} \Vert x(t) \Vert^2)&=& \langle \dot{x}(t), x(t) \rangle= \langle f(t,x(t))-By(t), x(t) \rangle\\
&\le& -\sigma \Vert x(t)\Vert^2 - \langle (B-C^T)y(t), x(t) \rangle- \langle y(t), Cx(t) \rangle\\
&\le& -\sigma \Vert x(t)\Vert^2 - \langle (B-C^T)y^{im}(t), x(t) \rangle + L_h\Vert y^{im}(t)\Vert \Vert x(t)\Vert-c_1 \Vert y^{im}(t)\Vert^2\\
&\le& -\sigma \Vert x(t)\Vert^2 - (\Vert B-C^T\Vert +L_h) \Vert \Vert y^{im}(t)\Vert \Vert x(t)\Vert-c_1 \Vert y^{im}(t)\Vert^2\\
&\le& -\delta \Vert x(t)\Vert^2.
\eaqn
Therefore
$$
\frac{d}{dt}(e^{2\delta t}\Vert x(t) \Vert^2) \le 0,
$$
which implies that
$$
\Vert x(t) \Vert\le e^{-\delta t} \Vert x_0 \Vert \to 0,\;\;{\rm as}\;\; t\to +\infty.
$$
\qed
\end{proof}
\begin{remark}
If only all the assumptions of Theorem \ref{wpmain} are satisfied, we may not have the uniqueness of solutions. However if the moving set always contains the origin, then all solutions starting at a given point $x_0\in \mathcal{A}$ also tend to zeros when the time is large. 
\end{remark}
\begin{theorem} (Globally exponential attractivity without uniqueness)\label{attractivity}
Suppose that all the assumptions of Theorem \ref{wpmain} are satisfied. Furthermore, assume that 
\beq
\langle f(t,x), x \rangle \le - \sigma \Vert x \Vert^2, \;\;0\in  K(t,x),\;\;\forall t\ge 0, x\in \R^n,
\eeq
for some $\sigma > \frac{ \Vert B-C^T \Vert^2}{4c_1}$. 
Then any solution  $x(\cdot)$ of $(\mathcal{S})$ starting at a given point $x_0$ exponentially converges to the origin when the time is large, i.e., 
$$
 \Vert x(t) \Vert \le e^{-\delta t} \Vert x_0 \Vert \to 0\;\;{\rm as}\;\; t\to +\infty,
$$
where $\delta:=\sigma - \frac{ \Vert B-C^T \Vert^2}{4c_1}>0$.
\end{theorem}
\begin{proof}
Similarly as in the proof of Theorem \ref{attractivity}, we know that for almost $t\ge 0$, one has 
$$
\dot{x}(t)= f(t,x(t))-By(t), 
$$
where 
$$
y(t) \in N_{K(t,x(t))}(Cx(t)-Dy(t)).
$$
The fact $0\in  K(t,x(t))$ deduces that 
$$
\langle y(t), 0-Cx(t)+Dy(t) \rangle\le 0.
$$
Thus 
$$
\langle y(t),Cx(t)\rangle\ge \langle y(t),Dy(t) \rangle\ge c_1  \Vert y^{im}(t)\Vert^2.
$$
Therefore 
\baqn
\frac{d}{dt}(\frac{1}{2} \Vert x(t) \Vert^2)&=& \langle \dot{x}(t), x(t) \rangle= \langle f(t,x(t))-By(t), x(t) \rangle\\
&\le& -\sigma \Vert x(t)\Vert^2 - \langle (B-C^T)y(t), x(t) \rangle- \langle y(t), Cx(t) \rangle\\
&\le& -\sigma \Vert x(t)\Vert^2 - \langle (B-C^T)y^{im}(t), x(t) \rangle -c_1 \Vert y^{im}(t)\Vert^2\\
&\le& -\sigma \Vert x(t)\Vert^2 - (\Vert B-C^T\Vert) \Vert \Vert y^{im}(t)\Vert \Vert x(t)\Vert-c_1 \Vert y^{im}(t)\Vert^2\\
&\le& -\delta \Vert x(t)\Vert^2.
\eaqn
Consequently, we have 
$$
 \Vert x(t) \Vert \le e^{-\delta t} \Vert x_0\Vert, \;\;\forall t\ge 0,
$$
and the conclusion follows. \qed
\end{proof}
\section{Application for studying time-varying Lur'e system with errors in data} \label{s4}
For simplicity, we consider the function $f$ as some  matrix $A$.
Suppose that  the matrices $ A, B, C, D$ and the time-varying set $K$ satisfy all the assumptions of Theorem 1. Then problem $(\mathcal{S})$ has  a unique solution.  However, assume that there are errors in measure for the matrices $A$ and $C$, i.e., we have the approximate matrices $\bar{A}, \bar{C}$ and we want to know whether the following system 
\begin{subequations}
\label{eq:tot}
\begin{empheq}[left={(\bar{{\mathcal S}})}\empheqlbrace]{align}
  & \dot{x}(t) = \bar{A}x(t)+{B}\lambda(t)\; {\rm a.e.} \; t \in [0,+\infty); \label{1a}\\
  & y(t)=\bar{C}x(t)+{D}\lambda(t),\\
  & \lambda(t) \in -N_{K(t)}(y(t)), \;t\ge 0;\\
  & x(0) = x_0,
\end{empheq}
\end{subequations}
has a solution. Generally, $(\bar{A}, B, \bar{C}, D, K)$ may not satisfy Assumptions $(A_1)-(A_5)$ so we can not apply the result in \cite{abc}. Let us show that we can use our result to answer this question. Indeed, we can rewrite $(\bar{{\mathcal S}})$ as follows
\begin{subequations}
\label{eq:tot}
\begin{empheq}[left={(\bar{{\mathcal S}})}\empheqlbrace]{align}
  & \dot{x}(t) = \bar{A}x(t)+{B}\lambda(t)\; {\rm a.e.} \; t \in [0,+\infty); \label{1a}\\
  & \bar{y}(t)=Cx(t)+{D}\lambda(t),\\
  & \lambda(t) \in -N_{K(t)}(\bar{y}(t)+(\bar{C}-C)x(t))=-N_{\bar{K}(t,x(t))}(\bar{y}(t)), \;t\ge 0;\\
  & x(0) = x_0,
\end{empheq}
\end{subequations}
where $\bar{K}(t,x)=K(t)-(\bar{C}-C)x$. Then the systems $(\bar{A}, B, {C}, D, K)$  satisfies all the assumptions of Theorem 1. Consequently, $(\bar{{\mathcal S}})$ has a solution defined on $[0,+\infty)$. If $ \bar{C}=C+\varepsilon(D+D^T)$ for some $\varepsilon>0$ small enough, then the solution is unique by using Theorem 2. In addition if $0\in K(t)$ for all $t\ge 0$, and $ \bar{A}\le -\sigma I$, i.e.,
$$
\langle \bar{A}x, x \rangle \le -\sigma \Vert x \Vert^2,\;\; \forall x\in \R^n,
$$
 for $\sigma>0$ is large enough, then the unique solution of $(\bar{{\mathcal S}})$ converges to the origin at exponential rate (Theorem 3).
 \begin{example}
 Let us consider 
 $$A=-\sigma I_2, B=D=\left( \begin{array}{cc}
0 &\;\; 0\\ \\
0 & \;\; 1
\end{array} \right),C=B+\varepsilon I_2,\;\;\;{K(t)=[f_1(t),+\infty)\times [f_2(t),+\infty)}
 $$
 \end{example}
for some $\sigma, \varepsilon>0$ where $f_1, f_2: [0,+\infty) \to \R$ are two absolutely continuous functions. Then there no exists a  positive symmetric matrix $P\in \R^{n\times n}$ such that 
 $$
 \ker(D+D^T)\subset \ker(PB-C^T).
 $$
 Therefore we can not apply the result in \cite{abc} but can use our result to deduce the existence of solutions for the associated dynamical system. Indeed, we can see that $(A, B,B,D,\bar{K})$ satisfies all assumptions of Theorem 1 where $\bar{K}(t,x)=[f_1(t)-\varepsilon x_1 ,+\infty)\times [f_2(t)-\varepsilon x_2,+\infty)$.
\begin{remark}
This application also suggests an idea to consider the time-varying Lur'e dynamical system when  $(A, B, C, D, K)$ does not satisfy Assumptions $(A_1)-(A_5)$ by modifying the matrix $C$ and reduce  the time-varying system into the state-dependent one.
\end{remark}

\section{Conclusions}
The paper studies the well-posedness and asymptotic behaviour  for a class of  Lur'e dynamical systems where the set-valued feedback depends not only on the time but also on the state. Let us emphasis that the obtained  solutions are strong, comparing with the weak solutions acquired in \cite{abc}. The main tool is a new implicit discretization scheme, which is an advantage for implementation in numerical simulations. Some conditions are given to obtain the exponential attractivity of the solutions. 
\label{sectionc}


\begin{thebibliography}{}
\bibitem{Acary}{\sc V. Acary, B. Brogliato}, {\em Numerical Methods for Nonsmooth Dynamical Systems. Applications
in Mechanics and Electronics}. Springer Verlag, LNACM 35, 2008.

 
 \bibitem{abc0} {\sc A. Tanwani, B. Brogliato, C. Prieur}  {\em Stability and observer design for Lur'e systems with
multivalued , non-monotone, time-varying
nonlinearities and state jumps}, SIAM J. Control Opti., Vol. 52, No. 6, pp. 3639--3672, 2014.
 
\bibitem{abc} {\sc A. Tanwani, B. Brogliato, C. Prieur}  {\em Well-Posedness and Output Regulation for Implicit Time-VaryingEvolution Variational Inequalities}, SIAM J. Control Opti., Vol. 56, No. 2, pp. 751--781, 2018.


\bibitem{ahl}  { \sc S. Adly, A. Hantoute, B. K. Le}, \emph{Nonsmooth Lur'e Dynamical Systems in Hilbert
Spaces}, Set-Valued Var. Anal, vol 24, iss. 1, 13-35, 2016.

\bibitem{ahl2}  { \sc S. Adly, A. Hantoute, B. K. Le}, \emph{Maximal Monotonicity and Cyclic-Monotonicity Arising in Nonsmooth Lur'e Dynamical Systems"}, Journal of Mathematical Analysis and Applications, 448 (2017), no. 1, 691--706 



\bibitem{al} {\sc S. Adly, B. K. Le}, {\em Stability and invariance results for a class of non-monotone set-valued Lur'e dynamical systems}, Applicable Analysis, vol. 93, iss.
5 (2014), 1087--1105.

\bibitem{Aubin-Cellina} {\sc J. P. Aubin, A. Cellina}, {\em Differential Inclusions. Set-Valued Maps and Viability Theory}, Spinger-Verlag, Berlin, 1984.


 




\bibitem{Brezis}{\sc H. Brezis}, {\em Op\'{e}rateurs Maximaux Monotones et Semi-groupes de Contractions dans
les Espaces de Hilbert}, Math. Studies 5, North-Holland American Elsevier (1973).

\bibitem{bdla} {\sc B. Brogliato, A. Daniliidis, C. Lemar\'echal, V. Acary}, {\em On the equivalence between
complementarity systems, projected systems and differential inclusions}, Systems
and Control Letters, vol. 55, no 1, pp. 45--51, January 2006.

\bibitem{bg} { \sc B. Brogliato} {\em Absolute stability and the Lagrange-Dirichlet theorem with monotone multivalued mappings}, Systems and Control Letters 2004, 51 (5), 343-353.

\bibitem{bg2}{\sc B. Brogliato, D. Goeleven},  {\em Well-posedness, stability and invariance results for a class of multivalued Lur'e dynamical systems}, Nonlinear Analysis: Theory, Methods and Applications, vol. 74, pp. 195--212, 2011.

\bibitem{bg1}{\sc B. Brogliato, D. Goeleven}, {\em Existence, uniqueness of solutions and stability of nonsmooth multivalued Lur'e dynamical systems}, Journal of Convex Analysis, vol. 20, no. 3, pp. 881--900, 2013.

\bibitem{brogliato}{\sc B. Brogliato, R. Lozano, B. Maschke, O. Egeland}, {\em Dissipative Systems Analysis and Control}, Springer-verlag London, 2nd Edition, 2007.

\bibitem{cs}{\sc M. K. Camlibel, J. M. Schumacher}, {\em Linear passive systems and maximal monotone mappings}, Math. Program. 157 (2), pp 397--420, 2016.

\bibitem{Cojocaru} {\sc M. G. Cojocaru, P. Daniele, A. Nagurney}, {\em Projected dynamical systems and evolutionary variational inequalities via Hilbert spaces and applications}, JOTA 127 No. 3 (2005), 549--563.
\bibitem{Grabowski} {\sc P. Grabowski, F. M. Callier}, {\em Lur'e feedback systems with both unbounded control and observation: Well-posedness and stability using nonlinear semigroups}, Nonlinear Analysis 74, 3065-3085, 2011.

\bibitem{Gwinner} {\sc J. Gwinner}, {\em On differential variational inequalities and projected dynamical systems - equivalence and a stability result}, Discrete and Continuous Dynamical Systems 2007, issue special, 467--476.

\bibitem{Gwinner1} {\sc J. Gwinner}, {\em On a new class of differential variational inequalities and a stability result}, Math. Prog. 2013, vol. 139, issue 1-2, 205--221.

\bibitem {Kunze} {\sc M. Kunze, M.D.P. Monteiro Marques}, BV solutions to evolution problems with time-dependent
domains, Set-Valued Anal. 5 (1997) 57--72.

\bibitem {Kunze1} 
{\sc M. Kunze, M.D.P. Monteiro Marques}, 
 {An introduction to Moreau's sweeping process }, 
  \emph{in ``Impacts in Mechanical Systems. Analysis
and Modelling"}, (B. Brogliato, Ed), 1-60, Springer, Berlin, 2000.


\bibitem{Le1} {\sc B. K. Le},  {\em Existence of Solutions for Sweeping Processes with
Local Conditions}, submitted. 
\bibitem{l}{\sc M. R. Liberzon}, {\em Essays on the absolute stability theory}, Automation and Remote
Control, vol. 67, no 10, pp. 1610--1644, October 2006.



%
\bibitem{Vladimirov} {\sc A. A. Vladimirov}, {\em Nonstationary dissipative evolution equations in a Hilbert space}, Nonlinear
Anal. 17 (1991) 499--518. 



\end{thebibliography}
\end{document}